\documentclass[12pt]{article}

\usepackage{graphics,amsmath,amssymb,amsthm,mathrsfs}

\newcommand{\rn}[1]{\mathbb{R}^{#1}}
\newcommand{\diverg}[1]{\mbox{div}\,#1\,}

\newtheorem{thm}{Theorem}[section]
\newtheorem{lemma}{Lemma}[section]
\newtheorem{cor}[thm]{Corollary}
\theoremstyle{definition}
\newtheorem{splemma}[lemma]{Lemma}
\newtheorem*{remark}{Remark}

\begin{document}
\bibliographystyle{amsplain}

\title{The ${L^p}$ Dirichlet Problem for the Stokes System on
Lipschitz Domains}
\author{Joel Kilty\thanks{The author is supported in part by the NSF (DMS-0500257).}}
\date{}
\maketitle

\begin{abstract}
    \noindent We study the $L^p$ Dirichlet problem for the Stokes system on
    Lipschitz domains.  For any fixed $p>2$, we show that a reverse
    H\"{o}lder condition with exponent $p$ is
    sufficient for the solvability of the Dirichlet problem with
    boundary data in $L^p_N(\partial\Omega,\rn{d})$.  Then we obtain a much simpler
    condition which implies the reverse H\"{o}lder condition.  Finally, we establish the solvability of
    the $L^p$ Dirichlet problem for $d\geq 4$ and
    $2-\varepsilon<p<\frac{2(d-1)}{d-3}+\varepsilon$.

\end{abstract}

\section{Introduction}

Let $\Omega$ be a bounded Lipschitz domain in $\rn{d}$, $d\geq 4$,
with connected boundary, and let $N$ be the outward unit normal to
$\partial\Omega$.  Then, set
$$L_N^p(\partial\Omega,\rn{d})=\left\{\vec{g}\in L^p(\partial\Omega,\rn{d}):
\int_{\partial\Omega} \vec{g}\cdot N\,d\sigma =0\right\}.$$ In this
paper we are interested in studying the $L^p$ Dirichlet problem for
the Stokes system:

\begin{equation} \label{stokesSystem}
    \left\{
    \begin{array}{ll}
        \Delta \vec{u} = \nabla p & \mbox{ in } \Omega, \\
        \diverg{\vec{u}} = 0 & \mbox{ in } \Omega, \\
        \vec{u}=\vec{f}\in L^p(\partial\Omega,\rn{d}) & \mbox{ on }
        \partial\Omega, \\
        (\vec{u})^*\in L^p(\partial\Omega), &
    \end{array}
    \right.
\end{equation}

\noindent where $(\vec{u})^*$ is the non-tangential maximal function
of $\vec{u}$ and the boundary values are taken in the sense of
non-tangential convergence.  Note that using the divergence theorem
we have the following necessary condition on the boundary data
$\vec{f}$:
$$\int_{\partial\Omega} \vec{f}\cdot N\,d\sigma =
\int_{\partial\Omega} \vec{u}\cdot N\,d\sigma = \int_{\Omega}
\diverg{\vec{u}}\,dx = 0,$$ i.e. $\vec{f}\in
L^p_N(\partial\Omega,\rn{d})$. We say that the $L^p$ Dirichlet
problem (\ref{stokesSystem}) on $\Omega$ is uniquely solvable if
given any $\vec{f}\in L^p_N(\partial\Omega,\rn{d})$, there exists a
unique function $\vec{u}$ and a unique function $p$, up to
constants, satisfying (\ref{stokesSystem}) where $\vec{u}=\vec{f}$
in the sense of non-tangential convergence, i.e.
$$\lim_{\substack{x\rightarrow Q \\ x\in \Gamma(Q)}} \vec{u}(x)
= \vec{f}(Q) \mbox{ for a.e. } Q\in \partial\Omega.$$  Here
$\Gamma(Q) = \left\{ x\in \Omega:
|x-Q|<2\mbox{dist}(x,\partial\Omega)\right\}$. Moreover, the
solution $\vec{u}$ satisfies
$\|(\vec{u})^*\|_{L^p(\partial\Omega,\rn{d})}\leq
C\|\vec{f}\|_{L^p(\partial\Omega,\rn{d})}$, where $C$ is independent
of the boundary data $\vec{f}$.

Since $\Omega$ is a bounded Lipschitz domain there exists $r_0>0$ such that for each
point $P\in\partial\Omega$ there is a new coordinate system of $\rn{d}$ obtained from the
standard Euclidean coordinate system through translation and rotation so that $P=(0,0)$ and
$$B(P,r_0)\cap \Omega = B(P,r_0)\cap \left\{ (x',x_d)\in \rn{d}: x'\in \rn{d-1} \mbox{ and }
x_d > \psi(x')\right\},$$ where $\psi:\rn{d-1}\rightarrow \rn{}$ is
a Lipschitz function and $\psi(0)=0$.  Throughout the paper we let

\begin{eqnarray*}
    \Delta(Q,r) &=& B(Q,r) \cap \partial\Omega, \\
    T(Q,r) &=& B(Q,r) \cap \Omega, \\
    I_r &=& \{ (x',\psi(x'))\in \rn{d-1}\times \rn{}:
    |x_1|<r,\dots |x_{d-1}|<r\}, \\
    Z_r &=& \{ (x',x_d): |x_1|<r,\dots, |x_{d-1}|<r,
    \psi(x')<x_d<C'r\},
\end{eqnarray*}

\noindent where $Q\in \partial\Omega$, $0<r<r_0$, and
$C'=1+10\sqrt{d}\|\nabla \psi\|_{\infty}>0$ is chosen so that $Z_r$
is a star-shaped Lipschitz domain with Lipschitz constant
independent of $r$.

The main results of this paper are as follows:

\begin{thm} \label{necessarySufficientCondition}
    Let $\Omega$ be a bounded Lipschitz domain with connected boundary in $\rn{d}$, $d\geq 4$, and $p>2$.  If there
    exists $C>0$ such that for any $Q\in
    \partial\Omega$ and $0<r<r_0$, the reverse H\"{o}lder
    condition,

    \begin{equation} \label{reverseHolder}
        \left(\frac{1}{|\Delta(Q,r)|}\int_{\Delta(Q,r)}|(\vec{u})^*|^p\,d\sigma\right)^{1/p}
        \leq C
        \left(\frac{1}{|\Delta(Q,2r)|}\int_{\Delta(Q,2r)}|(\vec{u})^*|^2\,d\sigma\right)^{1/2},
    \end{equation}

    \noindent holds for any solution $\vec{u}$ of the Stokes system
    (\ref{stokesSystem}) in $\Omega$ with the properties that
    $(\vec{u})^*\in L^2(\partial\Omega)$ and
    $\vec{u}=0$ on $\Delta(Q,3r)$, then the $L^p$ Dirichlet problem for the Stokes system
    (\ref{stokesSystem}) on $\Omega$ is uniquely solvable.
\end{thm}

\noindent It should be noted that nowhere in the proof of Theorem
\ref{necessarySufficientCondition} is the condition $d\geq 4$ used.
In the case $d=2,3$ the reverse H\"{o}lder Condition
(\ref{reverseHolder}) can be established for $2-\varepsilon < p <
\infty$, providing another proof of the solvability of the $L^p$
Dirichlet problem for $d=2,3$ and $2-\varepsilon<p < \infty$.

Next we establish a simpler condition which implies the reverse
H\"{o}lder condition given by estimate (\ref{reverseHolder}) using
the square function estimates as well as the regularity estimate.
This condition is given by the following theorem.

\begin{thm} \label{simplerConditionThm}
    Let $\Omega$ be a bounded Lipschitz domain with connected boundary in $\rn{d}$, $d\geq 4$.  Suppose that
    there exists a constant $C_1>0$ and $\lambda \in (0,d]$ such
    that for $0<r<R<r_0$ and $Q\in\partial\Omega$,

    \begin{equation} \label{simplerConditionEstimate}
        \int_{T(Q,r)} |\vec{u}|^2\,dx \leq C_1
        \left(\frac{r}{R}\right)^{\lambda} \int_{T(Q,R)}
        |\vec{u}|^2\,dx,
    \end{equation}

    \noindent whenever $\vec{u}$ is a solution of the Stokes system
    (\ref{stokesSystem}) in $\Omega$ with the properties that
    $(\vec{u})^*\in L^2(\partial\Omega)$ and $\vec{u}=0$ on
    $\Delta(Q,R)$.  Then, if $$2<p<2+\frac{4}{d-\lambda},$$ the $L^p$
    Dirichlet problem (\ref{stokesSystem}) is uniquely solvable.
\end{thm}

Finally, we establish condition (\ref{simplerConditionEstimate})
given in Theorem \ref{simplerConditionThm} to prove the following
corollary.

\begin{cor} \label{dirichletProblem}
    Let $\Omega$ be a bounded Lipschitz domain in $\rn{d}$, $d\geq
    4$, with connected boundary.  Then there exists $\varepsilon>0$,
    depending only on the Lipschitz character of $\Omega$ such that,
    given $\vec{f}\in L^p_N(\partial\Omega,\rn{d})$ with
    $2-\varepsilon < p < \frac{2(d-1)}{d-3} + \varepsilon$, the
    Dirichlet problem for the Stokes system (\ref{stokesSystem}) has
    a unique solution, and the solution satisfies the estimate
    $\|(\vec{u})^*\|_p \leq C\|\vec{f}\|_p$.
\end{cor}

For Lipschitz domains, the boundary value problems for Laplace's
equation with boundary data in $L^p$ are now well understood.  In
the late 1970's, Dahlberg \cite{dahlberg2,dahlberg1} showed that the
Dirichlet problem is uniquely solvable for $2-\varepsilon < p \leq
\infty$ where $\varepsilon$ depends on $d$ and the Lipschitz
character of $\Omega$.  Jerison and Kenig \cite{jerison} were then
able to solve the $L^2$ Neumann problem.  It should be noted that in
1984, Verchota \cite{verchota} established these same results by
extending the method of layer potentials to Lipschitz domains using
the celebrated theorem of Coifman, McIntosh, and Meyer
\cite{coifman} on the $L^2$ boundedness of Cauchy integrals on
Lipschitz curves.  Dahlberg and Kenig \cite{dahlberg} were then able
to show that the $L^p$ Neumann problem is uniquely solvable for
$1<p<2+\varepsilon$.

The method of layer potentials was also used to solve the $L^2$
Dirichlet problem for second order elliptic systems on Lipschitz
domains for $d\geq 3$ \cite{dahlberg3,fabes2,fabes,gao} as well as
the $L^2$ Dirichlet, Neumann, and regularity problems for the Stokes
system \cite{fabes}.  For $d=2,3$, Dahlberg and Kenig
\cite{dahlberg4} were able to show in 1990 that the $L^p$ Dirichlet
problem for elliptic systems is uniquely solvable for
$2-\varepsilon<p\leq\infty$ and the $L^p$ Neumann problem is
solvable for $1<p<2+\varepsilon$. Then, in 1995, Shen
\cite{shen:stokes} was able to establish the solvability of the
$L^{\infty}$ Dirichlet problem for the Stokes system when $d=3$ by
establishing certain decay estimates on the Green's functions thus
showing that when $d=3$, the Dirichlet problem is solvable for
$2-\varepsilon <p\leq\infty$.  In 2005, Shen \cite{shen:ellip} was
able to show that for $d\geq 4$ the $L^p$ Dirichlet problem for
elliptic systems was solvable for $2-\varepsilon < p <
\frac{2(d-1)}{d-3} + \varepsilon$. Also, in 2006, he established the
solvability of the $L^p$ Neumann problem for elliptic systems for
$p$ in the range $\frac{2(d-1)}{d+1}-\varepsilon<p<2+\varepsilon$
for $d\geq 4$ \cite{shen:boundary}.  The method used by Shen relies
on the solvability of the $L^2$ problem and certain reverse
H\"{o}lder inequalities. It is noted that these methods have also
been applied to the biharmonic equation \cite{dkv:biharmonic,
pv:biharmonic, shen:ellip,shen:biharmonic,verchota:biharmonic}.

In this paper, we use the following real variable argument proved by
Shen in \cite{shen:boundary} to establish a sufficient condition for
the solvability of the $L^p$ Dirichlet problem for the Stokes system
when $d\geq 4$ and $p>2$.

\begin{thm}\label{realVariableArgument}
    Let $S=\left\{ (x',\psi(x')):x'\in \rn{d-1}\right\}$ be a Lipschitz
    graph in $\rn{d}$. Let $Q_0$ be a surface cube in  $S$ and $F\in L^2(2Q_0)$.  Let
    $p>2$ and $g\in L^q(2Q_0)$ for some $2<q<p$.  Suppose that for
    each dyadic subcube $Q$ of $Q_0$ with $|Q|\leq \beta |Q_0|$,
    there exists two integrable functions $F_Q$ and $R_Q$ on $2Q$
    such that $|F|\leq |F_Q|+|R_Q|$ on $2Q$, and

    \begin{eqnarray*}
        \lefteqn{\left( \frac{1}{|2Q|} \int_{2Q} |R_Q|^p\,d\sigma
        \right)^{1/p} \leq}\\ && \qquad\qquad C_1\left\{\left(\frac{1}{|\alpha Q|} \int_{\alpha
        Q} |F|^2\,d\sigma\right)^{1/2} + \sup_{Q'\supset Q} \left(\frac{1}{|Q'|}
        \int_{Q'} |g|^2\,d\sigma\right)^{1/2}\right\}, \\
        \lefteqn{\left(\frac{1}{|2Q|}\int_{2Q} |F_Q|^2\,d\sigma\right)^{1/2}
        \leq
        C_2\sup_{Q'\supset Q} \left(\frac{1}{|Q'|} \int_{Q'}
        |g|^2\,d\sigma\right)^{1/2},}
    \end{eqnarray*}

    \noindent where $C_1,C_2>0$ and $0<\beta<1<\alpha$.  Then,

    \begin{eqnarray}
        \lefteqn{\left( \frac{1}{|Q_0|}\int_{Q_0} |F|^q\,d\sigma\right)^{1/q}
        \leq} \nonumber\\ &&  C_3\left\{\left(\frac{1}{|2Q_0|}\int_{2Q_0} |F|^2\,d\sigma\right)^{1/2} +
        \left(\frac{1}{|2Q_0|}\int_{2Q_0}
        |g|^q\,d\sigma\right)^{1/q} \right\}, \label{shenThmConclusion}
    \end{eqnarray}

    \noindent where $C_3$ depends only on $d$, $p$, $q$, $C_1$,
    $C_2$, $\alpha$, $\beta$, and $\|\nabla \psi\|_{\infty}$.
\end{thm}

\noindent We will then show that the condition given by estimate
(\ref{simplerConditionEstimate}) implies the reverse H\"{o}lder
condition.  Finally, condition (\ref{simplerConditionEstimate}) will
be established for some $\lambda >3$ to show that the Dirichlet
problem is uniquely solvable for $p$ in the range $2-\varepsilon < p
< \frac{2(d-1)}{d-3} + \varepsilon$. The proof will closely follow
the argument used by Shen in \cite{shen:necsuff}.

The paper is organized as follows: Theorem
\ref{necessarySufficientCondition} will be proved in section 2 and
Theorem \ref{simplerConditionThm} will be proved in section 3.
Finally, Corollary \ref{dirichletProblem} will be proved in section
4.

In preparation of this paper, the author learned of the work of Matt
Wright and Marius Mitrea \cite{wright} on the transmission problem
for the Stokes system.  As a corollary of their work they obtain a
proof of Corollary \ref{dirichletProblem}.  The proof provided in
this paper is a more direct approach to the problem.  Finally, the
author would like to acknowledge Zhongwei Shen for many very helpful
conversations.

\section{A Sufficient Condition}

The goal of this section is to prove Theorem
\ref{necessarySufficientCondition} which establishes a sufficient
condition for the solvability of the $L^p$ Dirichlet problem when
$p>2$ and $d\geq 4$.

\begin{proof}(of Theorem \ref{necessarySufficientCondition})\mbox{}

        First, note that the uniqueness for $p>2$ follows from the
        uniqueness for $p=2$.  Let $\vec{f}\in
        L^p_N(\partial\Omega,\rn{d})$ and let $\vec{u}$ be the solution
        of the $L^2$ Dirichlet problem with data $\vec{f}$.  We'll
        show that if $p>2$, then

        \begin{eqnarray}
            \left(\frac{1}{s^{d-1}}\int_{B(P,s)\cap \partial\Omega}
            |(\vec{u})^*|^p\,d\sigma\right)^{1/p} &\leq&
            C\left(\frac{1}{s^{d-1}}\int_{B(P,cs)\cap \partial\Omega}
            |(\vec{u})^*|^2\,d\sigma\right)^{1/2} \nonumber \\ && \qquad+
            C\left(\frac{1}{s^{d-1}}\int_{B(P,cs)\cap
            \partial\Omega}|\vec{f}|^p\,d\sigma\right)^{1/p},
            \label{ballReverse}
        \end{eqnarray}

        \noindent for any $P\in \partial\Omega$ and $0<s\leq cr_0$.
        Then, by covering $\partial\Omega$ with a finite number of
        balls of radius $cr_0$, estimate (\ref{ballReverse}) implies
        that

        \begin{eqnarray*}
            \|(\vec{u})^*\|_p &\leq&
            C_p\left(|\partial\Omega|^{\frac{1}{p}-\frac{1}{2}}\|(\vec{u})^*\|_2
            + \|\vec{f}\|_p\right) \\
            &\leq&
            C_p\left(|\partial\Omega|^{\frac{1}{p}-\frac{1}{2}}\|\vec{f}\|_2
            + \|\vec{f}\|_p \right) \\
            &\leq& C_p \|\vec{f}\|_p.
        \end{eqnarray*}

        \noindent Here we used the fact that $\vec{u}$ is the
        solution of the $L^2$ Dirichlet problem with data $\vec{f}$
        and H\"{o}lder's inequality.  It remains to establish
        estimate (\ref{ballReverse}); its proof relies on Theorem
        \ref{realVariableArgument}.

        Fix $Q\in \partial\Omega$ and $0<r<r_0$.  By rotation
        and translation we may assume that $Q=0$ and

        \begin{eqnarray*}
            B(0,cr_0) \cap \Omega &=& B(0,cr_0)\cap \{(x',x_d)\in
            \rn{d}: x_d>\psi(x') \}, \\
            B(0,cr_0) \cap \partial\Omega &=& B(0,cr_0) \cap \{
            (x',\psi(x')):x'\in\rn{d-1}\},
        \end{eqnarray*}

        \noindent where $\psi$ is a Lipschitz function on $\rn{d-1}$.

        Consider the surface cube $I_r$.  Write
        $\vec{u}=\vec{u}_1+\vec{u}_2$ where $\vec{u}_1$ is a solution to the $L^2$
        Dirichlet problem with boundary data $$\vec{f}_1=\left\{ \begin{array}{ll} \vec{f}
        & \hspace{0.1in} \mbox{ on } I_{8r} \\ \beta\vec{\alpha} & \hspace{0.1in} \mbox{ on }
        \partial\Omega\backslash I_{8r}, \end{array}\right.$$  where $\vec{\alpha}\in C^{\infty}(\rn{d},\rn{d})$
        is chosen so that $|\vec{\alpha}|\leq C_0$ and $\vec{\alpha}\cdot N\geq C_1>0$ (such a vector field has
        been shown to exist; for example, see the work of Verchota in \cite{verchota}).  Here $\beta$ is a constant depending
        on $\vec{f}$ chosen so that $\vec{f}_1\in
        L^2_N(\partial\Omega,\rn{d})$, i.e.

        \begin{equation} \label{beta}
            \beta =-
            \frac{\int_{I_{8r}} \vec{f}\cdot
            N\,d\sigma}{\int_{\partial\Omega\backslash I_{8r}} \vec{\alpha}\cdot
            N\,d\sigma}.
        \end{equation}

        \noindent Next, we verify the conditions of the real variable argument
        in Theorem \ref{realVariableArgument}.  Let

        \begin{eqnarray*}
            F &=& |(\vec{u})^*|, \\
            g &=& |\vec{f}|, \\
            F_Q &=& 2|(\vec{u}_1)^*|, \\
            R_Q &=& 2|(\vec{u}_2)^*|.
        \end{eqnarray*}

        \noindent Now, using the $L^2$ estimates for the Dirichlet problem we obtain

        {\allowdisplaybreaks
        \begin{eqnarray*}
            \frac{1}{|I_{2r}|}\int_{I_{2r}} |F_Q|^2\,d\sigma &\leq&
            \frac{C}{|I_{2r}|} \int_{\partial\Omega}
            |(\vec{u}_1)^*|^2\,d\sigma \\
            &\leq& \frac{C}{|I_{2r}|} \int_{\partial\Omega}
            |\vec{f}_1|^2\,d\sigma \\
            &\leq& \frac{C}{|I_{8r}|} \int_{I_{8r}}
            |\vec{f}|^2\,d\sigma +
            \frac{C|\beta|^2}{|I_{8r}|}\int_{\partial\Omega\backslash I_{8r}}
            |\vec{\alpha}|^2\,d\sigma \\
            &\leq& \frac{C}{|I_{8r}|}
            \int_{I_{8r}}|\vec{f}|^2\,d\sigma +
            C\frac{|\partial\Omega\backslash I_{8r}|}{|I_{8r}|}|\beta|^2 \\
            &\leq& \frac{C}{|I_{8r}|} \int_{I_{8r}} |\vec{f}|^2\,d\sigma +
            C\frac{|\partial\Omega\backslash I_{8r}|}{|I_{8r}|}\frac{|I_{8r}| \int_{I_{8r}} |\vec{f}|^2\,d\sigma}{\left|
            \int_{\partial\Omega\backslash I_{8r}} \vec{\alpha}\cdot
            N\,d\sigma \right|^2} \\
            &\leq& \frac{C}{|I_{8r}|} \int_{I_{8r}}
            |\vec{f}|^2\,d\sigma +
            \frac{C}{|\partial\Omega\backslash I_{8r}|}\int_{I_{8r}}|\vec{f}|^2\,d\sigma
            \\
            &\leq& \frac{C}{|I_{8r}|} \int_{I_{8r}}
            |\vec{f}|^2\,d\sigma
            \leq C\sup_{Q'\supset I_r} \frac{1}{|Q'|} \int_{Q'}
            |g|^2\,d\sigma.
        \end{eqnarray*}
        }%

        \noindent Note that $\vec{u}_2$ is a solution with $(\vec{u}_2)^*\in L^2(\partial\Omega)$ and
        $\vec{u}_2=0$ on $I_{8r}$.  Then, using
        the reverse H\"{o}lder Inequality
        (\ref{reverseHolder}) and the same estimates on $\vec{u}_1$
        as above, we obtain

        \begin{eqnarray}
            \lefteqn{\left( \frac{1}{|I_{2r}|}\int_{I_{2r}}
            |R_Q|^{p}\,d\sigma\right)^{1/p} \leq \left(
            \frac{C}{|I_{2r}|}\int_{I_{2r}}
            |(\vec{u}_2)^*|^p\,d\sigma\right)^{1/p}} \nonumber \\
            &\leq& C\left(\frac{1}{|I_{4r}|}\int_{I_{4r}}
            |(\vec{u}_2)^*|^2\,d\sigma\right)^{1/2} \nonumber\\
            &\leq& C \left( \frac{1}{|I_{4r}|}\int_{I_{4r}}
            |(\vec{u})^*|^2\,d\sigma\right)^{1/2} +
            C\left(\frac{1}{|I_{4r}|} \int_{I_{4r}}
            |(\vec{u}_1)^*|^2\,d\sigma \right)^{1/2}\nonumber \\
            &\leq& C\left(\frac{1}{|I_{4r}|}\int_{I_{4r}} |F|^2\,d\sigma\right)^{1/2}  +
            C\sup_{Q'\supset I_{r}} \left(\frac{1}{|Q'|} \int_{|Q'|}
            |g|^2\,d\sigma \right)^{1/2}. \label{RQEst}
        \end{eqnarray}

        \noindent Thus, both conditions of Theorem \ref{realVariableArgument} are
        satisfied and estimate (\ref{ballReverse}) is proven.
        Therefore, the solvability of the $L^p$ Dirichlet problem has been established.
\end{proof}

\begin{remark}
    In establishing estimate (\ref{RQEst}), the reverse H\"{o}lder
    condition was applied to a surface cube but was stated in
    Theorem \ref{necessarySufficientCondition} for surface balls.
    The reverse H\"{o}lder inequality for surface cubes can be
    obtained from that for surface balls by subdividing the surface
    cube into a finite number of smaller cubes that are contained in
    slightly larger surface balls.  This estimate can be made in
    such a way that the constant still depends only on the Lipschitz
    character of $\Omega$.
\end{remark}

\section{A Simpler and Stronger Sufficient Condition}

We begin this section by recalling a few known results.  The
following Caccioppoli's inequality  is contained in Theorem 2.2 of
\cite{giaquinta}.

\begin{lemma}[Caccioppoli's Inequality] \label{caccioppoli}
    Let $x_0 \in \overline{\Omega}$ and $r>0$ be small.  Assume that
    $(\vec{u},p)$ is a solution of the Stokes system
    (\ref{stokesSystem}) in $T(x_0,3r)$ such that $\vec{u}=0$ on
    $\Delta(x_0,3r)$.  Then,

    \begin{equation}
        \int_{T(x_0,r)} |\nabla \vec{u}|^2\,dx \leq
        \frac{C}{r^2}\int_{T(x_0,2r)} |\vec{u}|^2\,dx.
    \end{equation}

\end{lemma}

\noindent Then, Lemma \ref{caccioppoli} and standard arguments can
be used to prove the following lemma.

\begin{lemma}[Higher Integrability] \label{higherInt}
    Under the same assumptions as in Lemma \ref{caccioppoli} we have

    \begin{equation}
        \left(\frac{1}{r^d}\int_{T(x_0,r)} |\nabla
        \vec{u}|^q\,dx\right)^{1/q} \leq C
        \left(\frac{1}{r^d}\int_{T(x_0,2r)} |\nabla
        \vec{u}|^2\,dx\right)^{1/2},
    \end{equation}

    \noindent where $q>2$ depends only on the Lipschitz character of
    $\Omega$.
\end{lemma}

\noindent Finally, we recall the interior estimates for the Stokes
system.

\begin{lemma}[Interior Estimates]
    Let $\vec{u}$ be a solution of the Stokes System
    (\ref{stokesSystem}) in $\Omega$.  Suppose that $B(x,r)\subset
    \Omega$.  Then, $$|D^{\alpha} \vec{u}(x)| \leq
    \frac{C_{\alpha}}{r^{d+|\alpha|}} \int_{B(x,r)} |\vec{u}(y)|\,dy,$$ for any
    multi-index $\alpha$, where $C_{\alpha}$ depends only on
    $|\alpha|$ and $d$.
\end{lemma}

We will also need to use the square function estimates.  Recall that
the square function $S(w)$ is defined by

\begin{equation}
    S(w)(Q) = \left(\int_{\Gamma(Q)} \frac{|\nabla
    w(x)|^2}{|x-Q|^{d-2}}\,dx\right)^{1/2},
\end{equation}

\noindent for $Q\in \partial\Omega$.  We also define

\begin{equation}
    \tilde{S}(w)(Q) = \left(\int_{\Gamma(Q)} \frac{|\nabla^2
    w(x)|^2}{|x-Q|^{d-4}}\,dx \right)^{1/2},
\end{equation}

\noindent for $Q\in\partial\Omega$.  The following square function
estimates for solutions of the Stokes system (\ref{stokesSystem})
established in \cite{brown:stokes,dkpv:areaintegral} will be needed:

\begin{eqnarray}
    \left\|S( \vec{u})\right\|_{L^p(\partial\Omega)}
    &\leq& C \| (\vec{u})^*\|_{L^p(\partial\Omega)}, \label{sqFunct1}\\
    \left\| ( \vec{u})^*\right\|_{L^p(\partial\Omega)}
    &\leq& C \left\|
    S(\vec{u})\right\|_{L^p(\partial\Omega)} +
    C|\vec{u}(P_0)||\partial\Omega|^{1/p}, \label{sqFunct2}
\end{eqnarray}

\noindent where $0<p<\infty$, $P_0\in \Omega$, and $C$ depends on
the Lipschitz character of $\Omega$.  Then, using Lemma 2 on page
216 of \cite{stein:singularintegrals} along with the square function
estimate (\ref{sqFunct2}) we get that

\begin{equation} \label{nonTangentialTildeSquareEstimate}
    \|(\vec{u})^*\|_{L^p(\partial\Omega)} \leq C
    \|\tilde{S}(\vec{u})\|_{L^p(\partial\Omega)}  +
    C|\partial\Omega|^{\frac{1}{p}-\frac{1}{2}}\|(\vec{u})^*\|_{L^2(\partial\Omega)}.
\end{equation}

The following lemma found in \cite{shen:necsuff} is stated and
proved here for the sake of completeness.

\begin{lemma} \label{shenSquareTildeEstimate}
    Let $p>2$.  Then for any $\gamma \in (0,1)$ and $w\in
    C^2(\Omega)$ we have

    \begin{equation}
        \int_{\partial\Omega} |\tilde{S}(w)|^p\,d\sigma \leq
        C_{\gamma}\{\mbox{diam}(\Omega)\}^{\gamma} \int_{\Omega}
        |\nabla ^2w(x)|^p[\delta(x)]^{2p-1-\gamma}\,dx,
    \end{equation}

    \noindent where $\delta(x)=\mbox{dist}(x,\partial\Omega)$.
\end{lemma}

\begin{proof}
    We begin by re-writing $\tilde{S}(w)$ in the following manner
    $$\tilde{S}(w)(Q) = \left(\int_{\Gamma(Q)} \frac{|\nabla ^2
    w(x)|^2}{|x-Q|^{\frac{2(d+\gamma)}{p}-4}}\cdot
    \frac{dx}{|x-Q|^{\frac{(p-2)d-2\gamma}{p}}}\right)^{1/2}.$$
    Then, using H\"{o}lder's inequality we obtain

    \begin{eqnarray*}
        \tilde{S}(w)(Q) &\leq& C\left(\int_{\Gamma(Q)}
        \frac{|\nabla^2
        w(x)|^p}{|x-Q|^{d+\gamma-2p}}\,dx\right)^{1/p}\left(\int_{\Gamma(Q)}
        \frac{dx}{|x-Q|^{d-\frac{2\gamma}{p-2}}}\right)^{\frac{p-2}{2p}}
        \\
        &\leq& C \left(\int_{\Gamma(Q)} \frac{|\nabla^2
        w(x)|^p}{|x-Q|^{d+\gamma-2p}}\,dx\right)^{1/p}\left(\int_0^{\mbox{diam}(\Omega)}
        t^{\frac{2\gamma}{p-2}-1}\,dt \right)^{\frac{p-2}{2p}} \\
        &\leq&
        C\{\mbox{diam}(\Omega)\}^{\gamma/p}\left(\int_{\Gamma(Q)}
        \frac{|\nabla^2
        w(x)|^p}{|x-Q|^{d+\gamma-2p}}\,dx\right)^{1/p}.
    \end{eqnarray*}

    \noindent Finally, integrating $|\tilde{S}(w)(Q)|^p$ over
    $\partial\Omega$ we obtain

    \begin{eqnarray*}
        \int_{\partial\Omega} |\tilde{S}(w)(Q)|^p\,d\sigma &\leq& C
        \{\mbox{diam}(\Omega)\}^{\gamma} \int_{\partial\Omega}
        \int_{\Gamma(Q)} \frac{|\nabla^2
        w(x)|^p}{|x-Q|^{d+\gamma-2p}}\,dx\,d\sigma \\
        &\leq& C\{\mbox{diam}(\Omega)\}^{\gamma} \int_{\Omega}
        |\nabla^2w(x)|^p [\delta(x)]^{2p-\gamma-1}\,dx.
    \end{eqnarray*}
\end{proof}

\begin{splemma} \label{importantLemma}
    \emph{Let $p>2$.  Suppose that} $\Delta \vec{u} = \nabla p$
    \emph{and} $\mbox{div}(\vec{u})=0$ \emph{ in $\Omega$.  Then, for any $\gamma \in
    (0,1)$}

    \begin{eqnarray*}
        \int_{\partial\Omega} |(\vec{u})^*|^p\,d\sigma &\leq& C
        |\partial\Omega|^{1-\frac{p}{2}}\left(\int_{\partial\Omega} |(\vec{u})^*|^2\,d\sigma
        \right)^{p/2}  \\
        && \qquad + C_{\gamma} \{\mbox{diam}(\Omega)\}^{\gamma}
        \sup_{x\in \Omega} |\nabla^2
        \vec{u}|^{p-2}[\delta(x)]^{2p-2-\gamma} \int_{\partial \Omega}
        |(\nabla\vec{u})^*|^2\,d\sigma.
    \end{eqnarray*}
\end{splemma}

\begin{proof}
    Using Lemma \ref{shenSquareTildeEstimate} as well as the square
    function estimate (\ref{sqFunct1}) we obtain

    \begin{eqnarray*}
        \int_{\partial\Omega} |\tilde{S}(\vec{u})|^p\,d\sigma &\leq&
        C_{\gamma} \{\mbox{diam}(\Omega)\}^{\gamma} \int_{\Omega}
        |\nabla^2 \vec{u}|^p[\delta (x)]^{2p-1-\gamma}\,dx \\
        &\leq& C_{\gamma} \{\mbox{diam}(\Omega)\}^{\gamma} \sup_{x\in
        \Omega} |\nabla^2 \vec{u}|^{p-2}[\delta(x)]^{2p-2-\gamma}
        \int_{\Omega} |\nabla^2 \vec{u}|^2\delta(x)\,dx \\
        &\leq& C_{\gamma} \{\mbox{diam}(\Omega)\}^{\gamma}
        \sup_{x\in \Omega} |\nabla^2
        \vec{u}|^{p-2}[\delta(x)]^{2p-2-\gamma}\int_{\partial\Omega}
        |(\nabla \vec{u})^*|^2\,d\sigma.
    \end{eqnarray*}

    \noindent Combining this with estimate
    (\ref{nonTangentialTildeSquareEstimate}) we obtain

    \begin{eqnarray*}
        \int_{\partial\Omega} |(\vec{u})^*|^p\,d\sigma &\leq&
        C\int_{\partial\Omega} |\tilde{S}(\vec{u})|^p\,d\sigma +
        C|\partial\Omega|^{1-\frac{p}{2}}\left(\int_{\partial\Omega} |(\vec{u})^*|^2\,d\sigma
        \right)^{p/2} \\
        &\leq& C|\partial\Omega|^{1-\frac{p}{2}}\left(\int_{\partial\Omega}
        |(\vec{u})^*|^2\,d\sigma\right)^{p/2}
        \\
        && \qquad + C_{\gamma} \{\mbox{diam}(\Omega)\}^{\gamma}
        \sup_{x\in\Omega} |\nabla^2
        \vec{u}|^{p-2}[\delta(x)]^{2p-2-\gamma}\int_{\partial\Omega}
        |(\nabla \vec{u})^*|^2\,d\sigma.
    \end{eqnarray*}
\end{proof}

\begin{proof}(of Theorem \ref{simplerConditionThm}) \mbox{}

    Fix $\Delta(Q_0,r)$ with $Q_0\in \partial\Omega$ and $0<r<r_0$.
    Let $\vec{u}$ be a solution of the Stokes system
    (\ref{stokesSystem}) in $\Omega$ with the properties
    $(\vec{u})^*\in L^2(\partial\Omega)$ and $\vec{u}=0$ on
    $\Delta(Q_0,3r)$.  Now, using the assumption given by estimate
    (\ref{simplerConditionEstimate}), Caccioppoli's inequality, and the interior estimates, for
    any $x\in T(Q_0,r)$ we obtain

    \begin{eqnarray*}
        [\delta(x)]^2|\nabla^2 \vec{u}(x)| &\leq&
        C[\delta(x)]^2\left(\frac{1}{[\delta(x)]^d}\int_{B(x,c\delta(x))}
        |\nabla^2 \vec{u}(y)|^2\,dy\right)^{1/2} \\
        &\leq& \frac{C}{[\delta(x)]^{d/2}} \left(
        \int_{T(Q_0,\tilde{c}\delta(x))}
        |\vec{u}(y)|^2\,dy\right)^{1/2} \\
        &\leq& \frac{C}{[\delta(x)]^{d/2}} \left\{
        \left(\frac{\delta(x)}{r}\right)^{\lambda} \int_{T(Q_0,2r)}
        |\vec{u}(y)|^2\,dy \right\}^{1/2} \\
        &=& C \left(\frac{\delta(x)}{r}\right)^{\frac{\lambda-d}{2}}
        \left(\frac{1}{r^d}\int_{T(Q_0,2r)}
        |\vec{u}(y)|^2\,dy\right)^{1/2}.
    \end{eqnarray*}

    \noindent Thus, for $x\in T(Q_0,r)$ we have

    \begin{equation} \label{pointwiseSecondEstimate}
        |\nabla^2 \vec{u}(x)| \leq
        \frac{C}{[\delta(x)]^2}\left(\frac{\delta(x)}{r}\right)^{\frac{\lambda-d}{2}}
        \left(\frac{1}{r^d}\int_{T(Q_0,2r)} |\vec{u}(y)|^2\,dy
        \right)^{1/2}.
    \end{equation}

    By rotation and translation we may assume that $Q_0=0$ and
    $r_0=r_0(\Omega,d)>0$ is small enough so that

    \begin{eqnarray*}
        B(0,C_0r_0) \cap \Omega &=& B(0,C_0r_0)\cap \{(x',x_d) \in
        \rn{d-1}\times \rn{}: x_d> \psi(x')\}, \\
        B(0,C_0r_0)\cap \partial\Omega &=& B(0,C_0r_0) \cap \{
        (x',\psi(x')): x'\in \rn{d-1} \},
    \end{eqnarray*}

    \noindent where $\psi$ is a Lipschitz function on $\rn{d-1}$ and
    $\psi(0)=0$.  For $\rho\in (1,4)$ we also define

    \begin{eqnarray*}
        I_{\rho r} &=& \{ (x',\psi(x')): |x'|<\rho C_2 r\},\\
        Z_{\rho r} &=& \{ (x',x_d): |x'|<\rho C_2 r, \psi(x') < x_d
        < \psi(x')+\rho C_2 r \},
    \end{eqnarray*}

    \noindent where $C_2=C_2(d,\|\nabla \psi \|_{\infty})>0$ is
    small enough that $I_{3r}\subset \Delta(0,r)$ and $Z_{3r}\subset
    B(0,r)\cap \Omega$.  Let $\mathcal{M}_1$ and $\mathcal{M}_2$ be
    defined as follows:

    \begin{eqnarray*}
        \mathcal{M}_1(\vec{u})(Q) &=& \sup \{ |\vec{u}(x)|: x\in
        \Gamma(Q), |x-Q|<C_0 r\}, \\
        \mathcal{M}_2(\vec{u})(Q) &=& \sup \{ |\vec{u}(x)|: x\in
        \Gamma(Q), |x-Q|\geq C_0r \}.
    \end{eqnarray*}

    We begin by estimating $\mathcal{M}_2(\vec{u})$.  Using the
    interior estimates we obtain

    \begin{eqnarray*}
        |\vec{u}(x)| & \leq&  \frac{C}{|B(x,cr)|} \int_{B(x,cr)}
        |\vec{u}|\,dy \\
        &\leq& C\left( \frac{1}{r^{d-1}} \int_{\Delta(0,2r)}
        |(\vec{u})^*|^2\,d\sigma \right)^{1/2},
    \end{eqnarray*}

    \noindent for $x\in \Gamma(Q)$ such that $|x-Q|\geq cr$.  This
    implies that

    \begin{equation}
        \left(\frac{1}{r^{d-1}} \int_{I_r}
        |\mathcal{M}_2(\vec{u})|^p\,d\sigma \right)^{1/p} \leq C
        \left( \frac{1}{r^{d-1}} \int_{\Delta(0,2r)}
        |(\vec{u})^*|^2\,d\sigma \right)^{1/2}.
    \end{equation}

    Next we estimate $\mathcal{M}_1(\vec{u})$ on $I_r$ which is much
    more involved.  Applying Lemma \ref{importantLemma} to $\vec{u}$
    on the Lipschitz domain $Z_{\rho r}$ for $\rho \in
    (\frac{3}{2},2)$ we obtain

    \begin{eqnarray}
        \lefteqn{\frac{1}{r^{d-1}} \int_{I_r}
        |\mathcal{M}_1(\vec{u})|^p\,d\sigma \leq \frac{1}{r^{d-1}}
        \int_{\partial Z_{\rho r}} |(\vec{u})^*_{\rho}|^p\,d\sigma
        }\nonumber \\
        &\leq& C\left(\frac{1}{r^{d-1}} \int_{\partial Z_{\rho r}}
        |(\vec{u})^*_{\rho}|^2\,d\sigma \right)^{p/2} \label{m1Est}\\ && \qquad+
        C_{\gamma}r^{\gamma} \sup_{x\in Z_{\rho r}} |\nabla^2
        \vec{u}(x)|^{p-2} [\delta_{\rho}(x)]^{2p-2-\gamma}
        \frac{1}{r^{d-1}} \int_{\partial Z_{\rho r}} |(\nabla
        \vec{u})^*_{\rho}|^2\,d\sigma, \nonumber
    \end{eqnarray}

    \noindent where $\delta_{\rho}(x)=\mbox{dist}(x,\partial
    Z_{\rho r})$ and $(\nabla \vec{u})^*_{\rho}$ is the
    non-tangential maximal function of $\nabla \vec{u}$ with respect
    to the domain $Z_{\rho r}$.  Now, using the $L^2$ regularity
    estimate established by Fabes, Kenig, and Verchota in  \cite{fabes}
    and the fact that $\vec{u}=0$ on $\Delta(0,3r)$ we obtain

    \begin{eqnarray}
        \int_{\partial Z_{\rho r}} |(\nabla
        \vec{u})^*_{\rho}|^2\,d\sigma &\leq& C\int_{\partial Z_{\rho
        r}} |\nabla_t \vec{u}|^2\,d\sigma \nonumber \\
        &\leq& C\int_{\Omega\cap\partial Z_{\rho r}} |\nabla_t
        \vec{u}|^2\,d\sigma \leq C \int_{\Omega\cap\partial Z_{\rho r}} |\nabla
        \vec{u}|^2\,d\sigma. \label{nablaVEst}
    \end{eqnarray}

    Note that when $x\in Z_{\rho r}$ we have $\delta_{\rho}(x)\leq
    \delta(x) \leq Cr$.  Also note that the condition $$ p<2 +
    \frac{4}{d-\lambda}$$ implies that $$\frac{\lambda-d}{2}(p-2)
    +2>0.$$  Thus, we may choose $\gamma>0$ so small that
    $$\frac{\lambda-d}{2}(p-2) +2 -\gamma>0.$$  Thus, using estimate
    (\ref{pointwiseSecondEstimate})
    we have that

    \begin{eqnarray}
        \lefteqn{ r^{\gamma}\sup_{x\in Z_{\rho r}} |\nabla^2 \vec{u}(x)|^{p-2}
        [\delta_{\rho}(x)]^{2p-2-\gamma}} \nonumber \\  \hspace{0.2in} &\leq& C r^{\gamma}
        \sup_{x\in Z_{\rho r}}\frac{[\delta_{\rho}(x)]^{2p-2-\gamma}}{[\delta_{\rho}(x)]^{2(p-2)}}
        \left(\frac{\delta_{\rho}(x)}{r}\right)^{\frac{\lambda-d}{2}(p-2)}
        \left(\frac{1}{r^{d-1}} \int_{\Delta(0,2r)}
        |(\vec{u})^*|^2\,d\sigma\right)^{\frac{p-2}{2}}\nonumber \\
        &\leq& C r^{\gamma} \sup_{x\in Z_{\rho r}}
        [\delta_{\rho}(x)]^{2-\gamma} \left(\frac{1}{r^{d-1}}
        \int_{\Delta(0,2r)} |(\vec{u})^*|^2\,d\sigma
        \right)^{\frac{p-2}{2}}  \label{supTermEst}\\
        &\leq& Cr^2 \left(\frac{1}{r^{d-1}} \int_{\Delta(0,2r)}
        |(\vec{u})^*|^2\,d\sigma \right)^{\frac{p-2}{2}}. \nonumber
    \end{eqnarray}

    \noindent Combining estimates (\ref{m1Est}),(\ref{nablaVEst}), and (\ref{supTermEst}) we obtain

    \begin{eqnarray*}
        \lefteqn{\left(\frac{1}{r^{d-1}} \int_{I_r}
        |\mathcal{M}_1(\vec{u})|^p\,d\sigma \right)^{2/p}}\\ &\leq&
        \frac{C}{r^{d-1}} \int_{\Omega \cap \partial Z_{\rho r}}
        |\vec{u}|^2\,d\sigma \\ && \quad+ C \left(\frac{1}{r^{d-1}}
        \int_{\Delta(0,2r)} |(\vec{u})^*|^2\,d\sigma
        \right)^{\frac{p-2}{p}}\left(\frac{1}{r^{d-3}}\int_{\Omega\cap\partial
        Z_{\rho r}} |\nabla \vec{u}|^2\,d\sigma
        \right)^{2/p}.
    \end{eqnarray*}

    Now, using Young's inequality we get that

    \begin{eqnarray*}
        \left(\frac{1}{r^{d-1}} \int_{I_r}
        |\mathcal{M}_1(\vec{u})|^p\,d\sigma \right)^{2/p} &\leq&
        \frac{C}{r^{d-1}} \int_{\Omega \cap \partial Z_{\rho r}}
        | \vec{u}|^2\,d\sigma + \frac{C}{r^{d-1}} \int_{\Delta(0,2r)}
        |(\vec{u})^*|^2\,d\sigma  \\ && \qquad + \frac{C}{r^{d-3}}
        \int_{\Omega\cap\partial Z_{\rho r}} |\nabla
        \vec{u}|^2\,d\sigma.
    \end{eqnarray*}

    Integrating the above inequality in $\rho \in (\frac{3}{2},2)$
    and using the Caccioppoli inequality we obtain

    \begin{eqnarray*}
        \left(\frac{1}{r^{d-1}} \int_{I_r}
        |\mathcal{M}_1(\vec{u})|^p\,d\sigma \right)^{2/p} &\leq&
        \frac{C}{r^{d-1}} \int_{\Delta(0,2r)}
        |(\vec{u})^*|^2\,d\sigma + \frac{C}{r^d} \int_{Z_{2r}}
        |\vec{u}|^2\,dx  \\ && \qquad\qquad+ \frac{C}{r^{d-2}} \int_{Z_{2r}}
        |\nabla\vec{u}|^2\,dx \\
        &\leq&  \frac{C}{r^{d-1}} \int_{\Delta(0,2r)}
        |(\vec{u})^*|^2\,d\sigma + \frac{C}{r^d} \int_{Z_{4r}}
        |\vec{u}|^2\,dx \\
        &\leq& \frac{C}{r^{d-1}} \int_{\Delta(0,4r)}
        |(\vec{u})^*|^2\,d\sigma.
    \end{eqnarray*}

    \noindent Thus, a simple covering argument gives that
    $$\left(\frac{1}{r^{d-1}} \int_{\Delta(0,r)}
    |(\vec{u})^*|^p\,d\sigma \right)^{1/p} \leq C
    \left(\frac{1}{r^{d-1}} \int_{\Delta(0,4r)}
    |(\vec{u})^*|^2\,d\sigma \right)^{1/2}.$$  Thus, by Theorem
    \ref{necessarySufficientCondition} this implies the solvability
    of the $L^p$ Dirichlet problem on $\Omega$ for
    $$2<p<2+\frac{4}{d-\lambda}.$$

\end{proof}

\section{Solvability of the $L^p$ Dirichlet Problem}

    We conclude with the proof of Corollary \ref{dirichletProblem}.
    To prove the corollary we show that condition
    (\ref{simplerConditionEstimate}) is satisfied for some
    $\lambda>3$.

    \begin{proof}(of Corollary \ref{dirichletProblem})
        Let $\vec{u}$ be a solution of the Stokes system
        (\ref{stokesSystem}) with the properties $(\vec{u})^*\in
        L^2(\partial\Omega)$ and $\vec{u}=0$ on $\Delta(Q_0,R)$.
        By rotation and translation we may assume that $Q_0=0$ and
        use the notation in Theorem \ref{simplerConditionThm}.  Let
        $0<r<R/8$.  Note that

        \begin{eqnarray*}
            |\vec{u}(x)| &=&  |\vec{u}(x)-\vec{u}(Q)| \\
            &\leq& \int_0^{c\delta(x)} |\nabla \vec{u}|\,dx_d \\
            &\leq& C \delta(x) |(\nabla \vec{u})_{\rho}^*| \\
            &\leq& C r |(\nabla \vec{u})_{\rho}^*(Q)|,
        \end{eqnarray*}

        \noindent where $x=(Q,x_d)\in Z_r$ and $(\nabla \vec{u})_{\rho}^*$ is the
        non-tangential maximal function of $\nabla \vec{u}$ with respect to the
        Lipschitz domain $Z_{\rho R}$ for $\rho\in (\frac{1}{8},\frac{1}{4})$.  Thus,

        \begin{eqnarray*}
            \int_{Z_r} |\vec{u}|^2\,dx &\leq&
            C\int_{I_r}\int_{0}^{cr} r^2|(\nabla
            \vec{u})_{\rho}^*|^2\,dx_d d\sigma \\
            &\leq& Cr^3 \int_{I_r} |(\nabla \vec{u})_{\rho}^*|^2\,d\sigma
            \\
            &\leq& Cr^{3+(d-1)(1-\frac{2}{q})} \left(\int_{I_{\rho
            R}} |(\nabla \vec{u})_{\rho}^*|^q\,d\sigma
            \right)^{2/q},
        \end{eqnarray*}

        \noindent where $\rho \in (\frac{1}{8},\frac{1}{4})$ and $q>2$.  Now,
        we choose $q>2$ so that the regularity estimate holds uniformly
        on the Lipschitz domain $Z_{\rho R}$ for $\rho \in
        (\frac{1}{8},\frac{1}{4})$.  Then,

        \begin{eqnarray*}
            \left(\int_{Z_r} |\vec{u}|^2\,dx \right)^{q/2} &\leq& C
            r^{\frac{3q}{2}+(d-1)(\frac{q}{2}-1)} \int_{\partial
            Z_{\rho R}} |(\nabla \vec{u})_{\rho}^*|^q\,d\sigma \\
            &\leq& C r^{\frac{3q}{2}+(d-1)(\frac{q}{2}-1)}
            \int_{\partial Z_{\rho R}} |\nabla_t \vec{u}|^q\,d\sigma
            \\
            &\leq& C r^{\frac{3q}{2}+(d-1)(\frac{q}{2}-1)}
            \int_{\Omega\cap\partial Z_{\rho R}} |\nabla
            \vec{u}|^q\,d\sigma.
        \end{eqnarray*}

        \noindent Next, we integrate both sides of the above inequality in
        $\rho \in (\frac{1}{8},\frac{1}{4})$ to obtain

        \begin{equation}
            \left(\int_{Z_r} |\vec{u}|^2\,dx \right)^{q/2} \leq C r^{\frac{3q}{2}+(d-1)(\frac{q}{2}-1)}
            \frac{1}{R} \int_{Z_{\frac{R}{4}}} |\nabla \vec{u}|^q\,dx.
        \end{equation}

        \noindent Then, using Lemma \ref{higherInt} on the higher
        integrability and Caccioppoli's inequality we obtain

        \begin{eqnarray*}
            \int_{Z_r} |\vec{u}|^2\,dx &\leq& C
            r^{3+(d-1)(1-\frac{2}{q})} \frac{R^{2d/q}}{R^{2/q}}
            \left(\frac{1}{R^d} \int_{Z_{\frac{R}{4}}} |\nabla \vec{u}|^q\,dx
            \right)^{2/q} \\
            &\leq& C r^{3+(d-1)(1-\frac{2}{q})}
            R^{(d-1)(\frac{2}{q})} \frac{1}{R^{d+2}} \int_{Z_{R}}
            |\vec{u}|^2\,dx \\
            &\leq& C \left(
            \frac{r}{R}\right)^{3+(d-1)(1-\frac{2}{q})}
            \int_{Z_{R}} |\vec{u}|^2\,dx.
        \end{eqnarray*}

        \noindent Thus, condition (\ref{simplerConditionEstimate}) holds for
        $\lambda=3+(d-1)(1-\frac{2}{q})>3$.  Note that

        \begin{eqnarray*}
            2+\frac{4}{d-\lambda} &=& 2 +
            \frac{4}{d-3-(d-1)(1-\frac{2}{q})} \\
            &\geq& 2 + \frac{4}{d-3}= \frac{2(d-1)}{d-3}.
        \end{eqnarray*}

        \noindent Thus, the $L^p$ Dirichlet problem is solvable for
        $$2-\varepsilon < p <\frac{2(d-1)}{d-3} + \varepsilon. $$
    \end{proof}

\bibliography{stokesSystem}

\small
\noindent\textsc{Department of Mathematics,
University of Kentucky, Lexington, KY 40506}\\
\emph{E-mail address}: \texttt{jkilty@ms.uky.edu} \\

\end{document}